\documentclass[10pt]{amsart}
\usepackage{amssymb,amsmath,amsthm}
\usepackage[utf8]{inputenc}
\usepackage[italian,english]{babel}
\usepackage{mathrsfs}
\usepackage[all]{xy}
\usepackage{pgf,tikz}
\usetikzlibrary{arrows}
\usetikzlibrary{positioning,automata}
\usepackage{float}
\usepackage[labelformat=empty]{caption,subfig}
\usepackage{MnSymbol}
\usepackage{bm}
\usepackage{fixmath}

\usepackage[a4paper, total={4.5in, 7.125in}]{geometry}

\restylefloat{figure}

\theoremstyle{definition}
\newtheorem{defin}{Definition}[section]
\theoremstyle{plain}
\newtheorem{thm}[defin]{Theorem}

\newtheorem{cor}[defin]{Corollary}

\theoremstyle{remark}

\newcounter{num}

\title{Equitable block colourings}

\author{Paola Bonacini}
\email{bonacini@dmi.unict.it}

\author{Lucia Marino}
\email{lmarino@dmi.unict.it}

\address{Università degli Studi di Catania\\
  Viale A. Doria 6\\
95125 Catania\\
Italy}

\begin{document}

\begin{abstract}

Let $\Sigma=(X,\mathcal B)$ a $4$-cycle system of order $v=1+8k$. A $c$-colouring of type $s$ is a map $\phi\colon \mathcal B\rightarrow \mathcal C$,
with $C$ set of colours,  such that  exactly $c$ colours are used and for every vertex $x$ all the blocks containing $x$ are coloured exactly with $s$ colours. Let $4k=qs+r$, with $q,r\ge 0$. $\phi$ is \emph{equitable} if for every vertex $x$  the set of the $4k$ blocks containing $x$ is parted in $r$ colour classes of cardinality $q+1$ and $s-r$ colour classes of cardinality $q$. In this paper we study colourings for which $s|k$, giving a description of equitable block colourings for $c\in  \{s,s+1,\dots,\lfloor\tfrac{2s^2+s}{3}\rfloor \}$.
\end{abstract}

\maketitle

\section{Introduction}

Block colourings of $4$-cycle systems have been introduced and studied in \cite{GGR,GR}. For any vertex $x$ these colourings require particular conditions on the colours of the blocks containing $x$.

Let $K_v$ be the complete simple graph on $v$ vertices. The graph having vertices $a_1,a_2,\dots,a_k$, with $k\ge 3$, and having edges $\{a_k,a_1\}$ and  $\{a_i,a_{i+1}\}$ for $i=1,\dots,k-1$ is a $k$-cycle and it will be denoted by $(a_1,a_2,\dots,a_k)$. A $4$-cycle system of order $v$, briefly $4CS(v)$, is a pair $\Sigma=(X,\mathcal B)$, where $X$ is the set of vertices and $\mathcal B$ is a set of $4$-cycles, called \emph{blocks}, that partitions the edges of $K_v$. It is known that a $4CS(v)$ exists if and only if $v=1+8k$, for some $k\ge 0$.

A colouring of a $4CS(v)$ $\Sigma=(X,\mathcal B)$ is a mapping $\phi\colon \mathcal B\rightarrow \mathcal C$, where $\mathcal C$ is a set of colours. A $c$-colouring is a colouring in which exactly $c$ colours are used. The set of blocks coloured with a colour of $\mathcal C$ is a \emph{colour class}. A $c$-colouring of type $s$ is a colourings in which, for every vertex $x$, all the blocks containing $x$ are coloured exactly with $s$ colours. 

Let $\Sigma=(X,\mathcal B)$ a $4CS(v)$, with $v=1+8k$, let $\phi\colon \mathcal B\rightarrow \mathcal C$ be a $c$-colouring of type $s$ and let $4k=qs+r$, with $q,r\ge 0$. Note that each vertex of a $4CS(v)$, with $v=1+8k$, is contained in exactly $\tfrac{v-1}{2}=4k$ blocks. $\phi$ is \emph{equitable} if for every vertex $x$  the set of the $4k$ blocks containing $x$ is parted in $r$ colour classes of cardinality $q+1$ and $s-r$ colour classes of cardinality $q$. A bicolouring, tricolouring or quadricolouring is an equitable colouring with $s=2$, $s=3$ or $s=4$.

The colour spectrum of a $4CS(v)$ $\Sigma=(X,\mathcal B)$ is the set:
\[
\Omega_s(\Sigma)=\{c\mid \text{ there exists a an $c$-block-colouring of type $s$ of $\Sigma$}\}.
\] 
It is also considered the set $\Omega_s(v)=\bigcup \Omega_s(\Sigma)$, where $\Sigma$ varies in the set of all the $4CS(v)$. 

Let us recall that the \emph{lower $s$-chromatic index} is $\chi'_s(\Sigma)=\min \Omega_s(\Sigma)$ and the \emph{upper $s$-chromatic index} is $\overline{\chi}'_s(\Sigma)=\max \Omega_s(\Sigma)$. If $\Omega_s(\Sigma)=\emptyset$, then we say that $\Sigma$ is uncolourable. 

In the same way we consider $\chi'_s(v)=\min \Omega_s(v)$ and $\overline{\chi}'_s(v)=\max \Omega_s(v)$.

Block colourings for $s=2$, $s=3$ and $s=4$ have been studied in \cite{BM,GGR,GR}. The problem arose as a consequence of colourings of Steiner systems studied in \cite{CR,GHMR,GQ,V}.

In this paper we study colourings for which $s | k$, giving a description of equitable block colourings for $c\in  \{s,s+1,\dots,\lfloor\tfrac{2s^2+s}{3}\rfloor\}$.

\section{Main result}

In this section we prove the main result of the paper, giving in each case the construction of the desired colouring. In these constructions we will use the following symbolism. Let $A=\{a_1,a_2,\dots,a_{2p}\}$ and $B=\{b_1,b_2,\dots,b_{2q}\}$  be two sets such that $A\cap B=\emptyset$. We denote by $[A,B]$ the following family of $4$-cycles:
\[
[A,B]=\{(a_i,b_j,a_{i+p},b_{j+p})\mid 1\le i\le p,\, 1\le j \le q\}.
\]
Note that $|[A,B]|=pq$.

\begin{thm}
If $s\mid k$, then  $s,s+1,\dots,\lfloor\tfrac{2s^2+s}{3}\rfloor\in \Omega_s(v)$.
\end{thm}
\begin{proof}
Let $k=hs$, with $h\in \mathbb N$. Let $c\in\{s,s+1,\dots,\lfloor\binom{2s^2+s}{3}\rfloor\}$. 

\textbf{Case $\mathbf{c=s}$}. $\Sigma=(\mathbb Z_v,\mathcal B)$ with 
starter blocks $\{(0,i,4k+1,k+i)\mid 1\le i\le k\}$ is a $4$CS$(v)$. If we 
assign the colour $j$ to the blocks obtined for $i=(j-1)h+1,\dots,jh$ and all 
their translated forms, we get an $s$-block-colouring of type $s$ of $\Sigma$.  

\textbf{Case $\mathbf{c=s+1}$}. Consider the sets $A_i=\{a_{1}^{(i)},\dots,a_{8h}^{(i)}\}$ for $i=1,\dots,s$ 
with $A_i\cap A_j=\emptyset$ for $i\ne j$. Take $\infty \notin A_1\cup \dots 
\cup A_s$ and consider the following $4$-cycle systems of order $1+8h$:
\[
\Sigma_i=(A_i\cup \{\infty\},\mathcal B_{i})
\]
for $i=1,\dots,s$. We define the following $4$-cycle system $\Sigma=(X,\mathcal 
B)$ of order $v=1+s\cdot 8h=1+8k$, with:
\[
X=\bigcup_{i=1}^s A_i\cup \{\infty\}
\]
and
\[
	\mathcal B=\bigcup_{i=1}^s \mathcal B_i\cup \bigcup_{p<q}[A_p,A_q].
\]

We define a block-colouring $f\colon \mathcal B\rightarrow \{1,\dots,s+1\}$ as 
follows:
\begin{align*}
f(\medsquare)&=i\quad \forall \medsquare \in \mathcal B_i\\
f([A_p,A_q])&=i\quad  \forall p,q \mbox{ such that } p+q\equiv i \mod s+1.
\end{align*}
In this way we get a $s+1$-block-colouring of type $s$ of $\Sigma$.

\textbf{Case $\mathbf{s+2\le c\le \frac{s^2+s}{2}}$}. Take 
$p_1,\dots,p_{c-s-1},q_1,\dots,q_{c-s-1}\in \{1,\dots,s\}$ such that 
$(p_l,q_l)\ne (p_m,q_m)$ for $l\ne m$. We define a block-colouring 
\[
g\colon 
\mathcal B\rightarrow \{1,\dots,c\}
\]
as follows:
\begin{align*}
g(\medsquare)&=f(\medsquare)\quad \forall \medsquare \in \mathcal 
B\setminus\{[A_{p_1},A_{q_1}],\dots,[A_{p_{c-s-1}},A_{q_{c-s-1}}]\}\\
g([A_{p_i},A_{q_i}])&=s+1+i\quad  \forall i=1,\dots,c-s-1.
\end{align*}
In this way we get a $c$-block-colouring of type $s$ of $\Sigma$. 

\textbf{Case $\mathbf{\tfrac{s^2+s}{2}+1\le c \le \lfloor\tfrac{2s^2+s}{3}\rfloor}$}. 
Consider the sets $A_1,\dots,A_{2s}$ defined in the following way: 
\begin{itemize}
	\item $A_i=\{a_{1}^{(i)},\dots,a_{4h}^{(i)}\}$ for $i=1,3,\dots,2s-1$ 
	\item $A_{i+1}=\{a_{4h+1}^{(i)},\dots,a_{8h}^{(i)}\}$ for 
$i=1,3,\dots,2s-1$ 
	\item $A_i\cap A_j=\emptyset$ for $i\ne j$.
\end{itemize}
Take $\infty \notin A_1\cup \dots \cup A_{2s}$ and consider the following 
$4$-cycle systems of order $1+8h$:
\[
\Sigma_i=(A_i\cup  A_{i+1}\cup\{\infty\},\mathcal B_{i})
\]
for $i=1,3,\dots,2s-1$. Let us consider also the set:
\[
F= \{(p,q)\mid p,q=1,\dots,2s,\, p<q\}\setminus \{(1,2),(3,4),\dots,(2s-1,2s)\}.
\]
We define the following $4$-cycle system 
$\Sigma=(X,\mathcal B)$ of order $v=1+2s\cdot 4h=1+8k$, with:
\[
X=\bigcup_{j=1}^{2s} A_j\cup \{\infty\}
\]
and
\[
	\mathcal B=\bigcup_{j=0}^{s-1} \mathcal B_{2j+1}\cup \bigcup_{(p,q)\in F}[A_p,A_q].
\]
Let $c=\tfrac{s^2+s}{2}+t$, with $1\le t\le \tfrac{s^2-s}{6}$. Then: 
\begin{equation} \label{eq:1}
2s^2-2s=4t\cdot 3+\left(\tfrac{s^2-s}{2}-3t\right)\cdot 4.  
\end{equation}
Let us consider the complete graph $K_{2s}$ and the $1$-factor 
\[I= \{(1,2),(3,4),\dots,(2s-1,2s)\}.\]
By \eqref{eq:1} and by \cite[Theorem 2.4]{HHR} it is possible to divide $K_{2s}-I$ in $3$-cycles 
$C_1,\dots,C_{4t}$ and in $4$-cycles $C_{4t+1},\dots,C_{\tfrac{s^2-s}{2}+t}$. 

If $C_m=(i,j,k)$ is one of the $3$-cycles, 
let us consider 
\[\mathcal  C_m=\{[A_i,A_j],[A_j,A_k],[A_i,A_k]\};\] 
if $C_m=(i,j,k,l)$ is one of the $4$-cycles, let us consider 
\[\mathcal  C_m=\{[A_i,A_j],[A_j,A_k],[A_k,A_l],[A_l,A_i]\}.\] 
We define a 
block-colouring $f\colon 
\mathcal B\rightarrow \{1,\dots,c\}$ as follows:
\begin{align*}
f(\medsquare)&=\frac{i+1}{2}\quad \forall \medsquare \in \mathcal 
B_i\\
f([A_{p},A_{q}])&=s+i\quad  \forall [A_p,A_q]\in \mathcal C_i.
\end{align*}
In this way we get a $c$-block-colouring of type $s$ of $\Sigma$. 
\end{proof}

\begin{cor}
If $s|k$, then $\chi'_s(v)=s$.
\end{cor}

\begin{thm}
Let us suppose that $s\mid k$. Then $\overline{\chi}'_s(v)\le 
\tfrac{s^2v}{v+s-1}.$
\end{thm}
\begin{proof}
Let $\Sigma=(V,\mathcal B)$ be a 4CS$(1+8k)$ and let $\varphi\colon \mathcal 
B\rightarrow \mathcal C$ be a $c$-colouring of type $s$ of $\Sigma$ and let $k=hs$. Let 
$x\in \mathcal C$ and take element $v\in V$ incident with blocks of colour $x$. 
$v$ is contained in $4h$ blocks colour $x$ and so in $V$ there are at least 
$1+8h$ elements incident with blocks of colour $x$. Then $c(1+8h)\le sv$. 
\end{proof}

\end{document}